\documentclass[reqno]{amsart}

\usepackage{amsmath,amssymb,amscd, accents}
\usepackage{mathrsfs}
\usepackage[mathcal]{eucal}
\usepackage{esint}
\usepackage{enumitem}

\newtheorem{Thm}{Theorem}{\bfseries}{\itshape}
\newtheorem*{Thm*}{Theorem}{\bfseries}{\itshape}
\newtheorem{Cor}{Corollary}{\bfseries}{\itshape}
\newtheorem{Prop}[Cor]{Proposition}{\bfseries}{\itshape}
\newtheorem{Lem}[Cor]{Lemma}{\bfseries}{\itshape}
\newtheorem*{Lem*}{Lemma}{\bfseries}{\itshape}
\newtheorem{Fact}[Cor]{Fact}{\bfseries}{\itshape}
{\bfseries}{\itshape}
\newtheorem{Def}[Cor]{Definition}{\bfseries}{\rmfamily}
\newtheorem{Ex}[Cor]{Example}{\scshape}{\rmfamily}
\newtheorem{Rem}[Cor]{Remark}{\scshape}{\rmfamily}
{\bfseries}{\itshape}

\renewcommand\ge{\geqslant} \renewcommand\le{\leqslant}
\let\tildeaccent=\~ \let\hataccent=\^
\renewcommand\~[1]{\widetilde{#1}}

\def\<{\left<} \def\>{\right>} \def\({\left(} \def\){\right)}

\def\ac{{\overline{\rm ac}}}
\def\FF{{\mathbb F}}
\def\AA{{\mathbb A}}
\def\llp{\mathopen{(\!(}}

\def\rrp{\mathopen{)\!)}}

\let\polishL=l \def\Zoladek.{\.Zol\c adek}

 \def\Im{\operatorname{Im}}
 
 \def\ord{\operatorname{ord}}
 
\def\etc.{\emph{etc}.}

\def\:{\colon} \def\R{{\mathbb R}} \def\C{{\mathbb C}} \def\Z{{\mathbb
    Z}} \def\N{{\mathbb N}} \def\Q{{\mathbb Q}}

\def\A{{\mathbb A}}

\let\PolishL=\L 
\def\L{{\mathbb L}}

 \def\e{\varepsilon}

\def\poly{{\operatorname{poly}}}

 \def\d{\,\mathrm d}

 \def\Lojas.{\PolishL ojasiewicz}

 \def\cL{{\mathcal L}}

\def\cO{{\mathcal O}}

\def\rest#1{{\vert_{#1}}}

\def\supp{\operatorname{supp}}

\def\alg{\mathrm{alg}}
\def\trans{\mathrm{trans}}

\def\vf{{\mathbf f}}
\def\vg{{\mathbf g}}
\def\vx{{\mathbf x}}

\def\vz{{\mathbf z}}
\def\vw{{\mathbf w}}
\def\vp{{\mathbf p}}

\begin{document}

\title[Point counting and Wilkie's conjecture over function fields]
{Point counting and Wilkie's conjecture for non-archimedean Pfaffian and Noetherian functions}

\author{Gal Binyamini}
\address{Weizmann Institute of Science, Rehovot, Israel}
\email{gal.binyamini@weizmann.ac.il}

\author{Raf Cluckers}
\address{Univ.~Lille,
CNRS, UMR 8524 - Laboratoire Paul Painlev\'e, F-59000 Lille, France, and,
KU Leuven, Department of Mathematics, B-3001 Leu\-ven, Bel\-gium}
\email{Raf.Cluckers@univ-lille.fr}
\urladdr{http://rcluckers.perso.math.cnrs.fr/}

\author{Dmitry Novikov}
\address{Weizmann Institute of Science, Rehovot, Israel}
\email{dmitry.novikov@weizmann.ac.il}

\thanks{The authors would like to thank the Fields Institute where the first ideas of this paper emerged, and Florent Martin and Arne Smeets for input during preliminary stages of the project.
R.\,C. was partially supported by the European Research Council under the European Community's Seventh Framework Programme (FP7/2007-2013) with ERC Grant Agreement nr. 615722 MOTMELSUM, KU Leuven IF C14/17/083, and Labex CEMPI  (ANR-11-LABX-0007-01). This research was partially supported by the ISRAEL SCIENCE FOUNDATION (grant No. 1167/17) and by funding received from the MINERVA
Stiftung with the funds from the BMBF of the Federal Republic of  Germany. This project has received funding from the European Research Council (ERC) under the European Union's Horizon 2020 research and innovation programme (grant agreement No 802107). The authors would like to express their gratitude to the anonymous referees for
several comments improving the accuracy and readability of the text}


\begin{abstract}
  We consider the problem of counting polynomial curves on analytic or
  definable subsets over the field $\C\llp t \rrp$, as a function of
  the degree $r$. A result of this type could be expected by analogy
  with the classical Pila-Wilkie counting theorem in the archimedean
  situation.

  Some non-archimedean analogs of this type have been developed in the
  work of Cluckers-Comte-Loeser for the field $\Q_p$, but the
  situation in $\C\llp t \rrp$ appears to be significantly
  different. We prove that the set of polynomial curves of a fixed
  degree $r$ on the transcendental part of a subanalytic set over
  $\C\llp t \rrp$ is automatically finite, but give examples showing
  that their number may grow \emph{arbitrarily quickly} even for analytic
  sets. Thus no analog of the Pila-Wilkie theorem can be expected to
  hold for general analytic sets. On the other hand we show that if
  one restricts to varieties defined by Pfaffian or Noetherian
  functions, then the number grows at most \emph{polynomially} in $r$,
  thus showing that the analog of the Wilkie conjecture does hold in
  this context.
\end{abstract}
\date{\today}
\maketitle

\section{Introduction}

\subsection{Point counting in archimedean and non-archimedean fields}

Over the reals, it is known by bounds of Pila and Wilkie \cite{PW}
that the number of rational points of height at most $B$ on the
transcendental part of analytic varieties (or even definable sets in
o-minimal structures) is bounded above by $cB^\varepsilon$ for some
$c=c(\varepsilon)$ and with $\varepsilon>0$. Such bounds also hold
over $\Q_p$ by \cite{CCL-PW}, and, with uniformity in $p$ by
\cite{CFL}. However, for $\C\llp t \rrp$, the question about
appropriate upper bounds for rational points on analytic (definable)
sets is left open in \cite{CCL-PW}. In this context, it is natural to
count points from $\C[t]$ lying in the transcendental part of a
definable set, as a function of the degree $r$.
The question of the finiteness of the set of such polynomials of bounded degree
is discussed in \cite[Section 5.5]{CCL-PW},
as well as the possibility of Wilkie type bounds under
some extra Pfaffian style conditions, but the methods of loc.~cit.~did not allow to establish any bounds on these sets, not even their finiteness.

\subsection{General definable sets: finiteness and a negative result}

In the first part of this paper we consider the analog of the
Pila-Wilkie counting theorem in the context of the field
$\C\llp t\rrp$. We make two contributions in this direction. First, we
show that the set under consideration is indeed finite, so that the
counting problem is well-posed. Second, we produce examples of
analytic sets where the number of such polynomial curves grows
arbitrarily fast as a function of $r$. In other words, no analog of
the Pila-Wilkie counting theorem can be expected in this context.  We
also give a variant of the finiteness result over $\Q_p^{\rm unram}$,
the maximal unramified field extension of $\Q_p$. These results are
presented in Section~\ref{sec:finiteness}.

It is easy to explain intuitively why it may be unreasonable to expect
a Pila-Wilkie type counting result over the field $\C\llp t\rrp$. The
basic idea behind the Bombieri-Pila method and its subsequent
generalization in the work of Pila-Wilkie is that after making a
suitable parametrization and cutting the domain into balls of radius
$cB^{-\varepsilon}$, the rational points in each ball can be
interpolated by a hypersurface of degree $d=d(\e)$. A similar strategy
has been made to work in \cite{CCL-PW} for the field $\Q_p$. On the
other hand, in $\C\llp t\rrp$ the valuation field is infinite, and it
is simply not possible to subdivide a ball of radius $1$ into any
finite collection of smaller balls. The entire approach, it appears,
is doomed to fail --- and this is borne out by our counterexamples.

\subsection{Wilkie's conjecture}

The finiteness result, and the impossibility of Pila-Wilkie type
bounds in general, serves as double motivation to search for a
framework with additional control where some results in the spirit of
the counting theorem can still be obtained. A potent source of
intuition in this direction is the Wilkie conjecture. This prominent
conjecture due to Wilkie states that for certain natural o-minimal
structures (originally $\R_{\exp}$ in Wilkie's formulation) one can
sharpen the bound $c(\e) B^\e$ to some polynomial in $\log B$. Various
special cases of the Wilkie conjecture have been established for sets
defined using Pfaffian functions: either in small dimensions
\cite{CPW,JT,Pila:Mild}, or for general sets definable in the class of
``holomorphic-Pfaffian'' functions \cite{me:rest-wilkie}.

Given that in $\C\llp t\rrp$ we are unable to subdivide the unit ball
into any number of smaller balls, the only hope seems to be to show
that for some suitable degree $d$, the rational points can all be
approximated \emph{without any subdivision}, i.e. using the single
unit ball. According to work on the Wilkie conjecture in the
archimedean context, it is known that using hypersurfaces of degree
$d=(\log B)^\alpha$ (rather than $d=d(\e)$), it is in fact possible to
interpolate all rational points using $(\log B)^\beta$ balls (rather
than $cB^{-\varepsilon}$). In some cases, even $O(1)$ balls
suffice. It therefore appears at least potentially plausible that in
the more restrictive context of the Wilkie conjecture, the point
counting argument can be salvaged.

Following this intuition, we introduce the Pfaffian and Noetherian
functions into the non-archimedean picture.  In the second part of the
paper with Sections~\ref{sec:wilkie-results} -- \ref{sec:wilkie-proofs} we consider an
analog of the Wilkie conjecture, with proof techniques which are independent of Section \ref{sec:finiteness}. Namely, we restrict attention to the
class of germs of sets defined by \emph{Pfaffian} or \emph{Noetherian}
equations over the field of convergent Laurent series
$\C(\!\{t\}\!)$. Many functions of interest in the classical applications
of the Pila-Wilkie theorem (e.g.~abelian functions, modular functions,
period integrals) fall within these classes, and this therefore seems
like a natural context in which to pursue non-archimedean counting
theorems. By the general philosophy of the Wilkie conjecture one
expects sharper, even polynomial in $r$, bounds for the counting
problem on such sets. Surprisingly, we show that despite the failure
of the Pila-Wilkie counting theorem in this context, the analog of the
Wilkie conjecture does in fact hold for arbitrary Pfaffian, or even
Noetherian, varieties. Specifically, we show that the subdivision step
can be completely avoided in this case, and a single hypersurface of
degree polynomial in $\log B$ can indeed be used to interpolate all
rational points of height $B$ in the unit ball.
We remark that  this  second part of the paper can be read independently of Section~\ref{sec:finiteness}.

\subsection{Main ideas}

We interpret systems of equations over $\C(\!\{t\}\!)$ as one-parameter
deformations of systems over $\C$. A crucial technical difference
arises when comparing this local context to the archimedean one. In
the archimedean context, all known cases of the Wilkie conjecture rely
in a crucial way on Khovanskii's B\'ezout type bounds for Pfaffian
functions over the reals \cite{Khov}, which gives bounds for the
number of solutions of systems of Pfaffian equations in terms of their
degrees. This theory is \emph{purely real}, and generally gives bounds
only for real solutions of systems of equations. On the other hand, in
the local $\C(\!\{t\}\!)$-context, general bounds are in fact available
for \emph{complex}, rather than real, solutions. Indeed, Gabrielov
established general bounds for the number of solutions of such
complex-analytic deformations of Pfaffian functions in
\cite{Gab:MultPfaffian}. Moreover, under a small technical
restriction, similar results have been established in the class of
Noetherian functions by Binyamini and Novikov in
\cite{me:deflicity}. Nothing approaching such general results is known
in the archimedean situation. Using these tools, it is at least
plausible to expect that one can treat the case of general Pfaffian or
even Noetherian functions by the complex-analytic methods introduced
in \cite{me:rest-wilkie}.

In Sections~\ref{sec:weierstrass-polydiscs}--~\ref{sec:wilkie-proofs}
we carry out this program. We introduce a local analog of the
Weierstrass polydiscs used in \cite{me:rest-wilkie}, and apply the
results of \cite{Gab:MultPfaffian,me:deflicity} to show that these can
be constructed with appropriate control over complexity for Pfaffian
and Noetherian varieties. It is pleasantly surprising that in this
local context we can achieve this in full generality, for both
Pfaffian and Noetherian functions, whereas the corresponding results
in the archimedean context are currently far more restricted in
scope. Let us finally mention that this seems to be the first study of
Pfaffian and Noetherian functions in a non-archimedean context, as far
as we can see. We hope that this will open the way for the study of
Pfaffian and Noetherian functions over $\Q_p$ instead of
$\C\llp t \rrp$.


\section{Finiteness results}
\label{sec:finiteness}

In this section we state and prove our finiteness result for rational
points of bounded height (i.e.~with coordinates which are polynomials
in $\C[t]$, of bounded degree) on the transcendental part of analytic
definable sets (see Theorem \ref{thm:finite}), and, we show that these
numbers can grow arbitrarily fast with the degree (see Proposition
\ref{thm:fast-in-r}). At the end of Section \ref{sec:finiteness}, we adapt the finiteness result to the mixed characteristic case.



Let us make this all very precise. In this section we write $K$ for
$\C\llp t \rrp$.
We recall some of the notions for $K=\C\llp t \rrp$
from \cite{CCL-PW}, analogous to the corresponding real notions. For a
set $X\subset K^n$, let $X^{\rm alg}$ be the union of all
semi-algebraic sets $C\subset X$ which are of constant local dimension
$1$. Here, semi-algebraic means
definable with constants from $K$ in the language of valued fields,
with symbols $+,-,\cdot,|$,
where $x|y$ holds for $(x,y)$ in $K^2$ if and only if
$y$ lies in $x\cO_K$  and where  $\cO_K:= \C[[ t ]]$ is the valuation ring of $K$.
The local dimension of a nonempty
semi-algebraic set $C\subset K^n$ at $x\in C$ is defined to be the
maximal integer $m\geq 0$ such that for all sufficiently small
semi-algebraic open neighborhoods $U$ of $x$ there is a semi-algebraic
function $U\cap C\to K^m$ whose range has nonempty interior in $K^m$.
Correspondingly, the transcendental part $X^{\rm trans}$ of $X$ is
defined as $X\setminus X^{\rm alg}$.

As replacement for the o-minimality condition in \cite{PW}, we will
impose a form of analyticity on $X$, as follows. For each integer
$n\geq 0$, let $\cO_K\langle x_1,\ldots,x_n\rangle$ be the $t$-adic
completion of $\cO_K[x_1,\ldots,x_n]$ inside
$\cO_K[[x_1,\ldots,x_n]]$, for the Gauss norm. Note that
$\cO_K\langle x_1,\ldots,x_n\rangle$ consists of power series
$\sum_{i\in\N^n} a_i x^i$, in multi-index notation, with
$a_i\in \cO_K$ and such that the $t$-adic norm $|a_i|$ of $a_i$ goes
to zero when $i_1+\ldots + i_n$ goes to $+\infty$. Here, the $t$-adic
norm $|x|$ of nonzero $x\in K$ is defined as $e^{-\ord x}$ where
$\ord x$ is the $t$-adic valuation of $x$, namely, the largest integer
$n$ such that $x/t^n$ lies in $\cO_K$, the $t$-adic norm of $0$ is
defined to be $0$, and for $x\in K^n$, one defines $|x|$ as the
maximum of the $|x_i|$ for $i=1,\ldots,n$. For $f$ in
$\cO_K\langle x_1,\ldots,x_n\rangle$, write $f^|$ for the restricted
analytic function associated to $f$, namely, the function $K^n\to K$
sending $z\in \cO_K^n$ to the evaluation $f(z)$ of $f$ at $z$ (namely,
the limit for the $t$-adic topology over $s>0$ of the partial sums
$\sum_{i,\ \max_j i_j < s} a_i z^i$), and sending the remaining $z$ to
$0$.  Let $\cL_{\rm an}^K$ be the language containing the language of
valued fields and, for each $f$ in
$\cO_K\langle x_1,\ldots,x_n\rangle$ for any $n\geq 0$, a function
symbol for the restricted analytic function $f^|$ associated to $f$.

Finally, we follow \cite{CCL-PW} also for the notion of integral
points of bounded height on subsets of $K^n$, which we now recall.
For $r\geq 1$, denote by $\cO_{K}(r)$ the subset of $\cO_K$ consisting
of polynomials $\sum_{i=0}^{r-1} a_i t^i$ with $a_i$ in $\C$ and with
degree less than $r$ (in the variable $t$).  For any subset
$X\subset K^n$ and any $r\geq 1$, write $X(r)$ for the intersection of
$X$ with $(\cO_{K}(r))^n$.

We can now state the first main result of this paper, addressing a question
left open in Section 5.5 of \cite{CCL-PW} (see the partial result,
Proposition 5.5.1 of \cite{CCL-PW}).

\begin{Thm}[Finiteness]\label{thm:finite}
Let $X\subset K^n$ be $\cL_{\rm an}^K$-definable. Then, for each $r>0$,
$$
(X^{\rm trans})(r)
$$
is a finite set.
\end{Thm}

Of course one would like to bound $\# X^{\rm trans}(r)$ when $r$
grows. However, without extra information about the geometry of $X$,
it is hard to bound the number of points in $X^{\rm trans}(r)$, as
$X^{\rm trans}(r)$ can grow arbitrarily fast with $r$ by the following
result.
\begin{Prop}\label{thm:fast-in-r}
  For any sequence of positive integers $N_r$ for $r>0$, there is an
  $\cL_{\rm an}^K$-definable set $X\subset K^2$ such that
  $$
  N_r < \#(X^{\rm trans})(r) .
  $$
  Furthermore, for $X$ one can even take the graph of a function
  $f:\cO_K\to\cO_K$ given by a power series in
  $\cO_K\langle  x \rangle $, in one variable $x$.
\end{Prop}


The proof of Theorem \ref{thm:finite} uses a quantifier elimination
result in a certain expansion $\cL_{{\rm an}}^\ac$ of $\cL_{\rm an}^K$
and a reduction to Zariski constructible conditions on tuples of
complex polynomials in $t$ of degree less than $r$ (some similar
techniques appear in \cite{BN1}, \cite{me:rest-wilkie},
\cite{CCL-PW}). We start with developing these ideas, summarized in
Proposition \ref{prop:constr} which gives that $X(r)$ is a
constructible set. Let us mention that languages and formulas are
always first order in this paper (as is typical in model theory).

Consider the language $\cL_{\rm an}^\ac$ with three sorts
(respectively valued field, residue field and value group), containing
$\cL_{\rm an}^K$ together with the field inverse $(\cdot)^{-1}$ extended by zero on zero for the valued field sort, the ring language on the
residue field, the Presburger language on the value group, an angular
component map $\ac$ sending nonzero $x$ in $K$ to the coefficient of
the leading term in $t$ of $x$ (namely to $xt^{-\ord x}\bmod (t)$),
and sending zero to zero, and the valuation map from $K^\times$ to the
value group.  Clearly any $\cL_{\rm an}^K$-definable set is also
$\cL_{\rm an}^\ac$-definable.

By the quantifier elimination statement of Theorem 4.2 of \cite{CLR}
together with quantifier elimination in the Presburger language and
the Chevalley-Tarski Theorem, any $\cL_{\rm an}^\ac$-definable set $X$
is given by a quantifier free formula in the language
$\cL_{\rm an}^\ac$. (This also follows from Theorem (3.9) of
\cite{vdDAx}, or, Theorem 6.3.7 and Example 4.4(1) from \cite{CLip}.)

We can now state and prove the reduction to Zariski constructible conditions.

\begin{Prop}\label{prop:constr}
  Let $X\subset K^n$ be $\cL_{\rm an}^\ac$-definable and let $r>0$ be
  an integer. Then $X(r)$ is a Zariski constructible subset of
  $\C^{rn}$, where we identify $(\cO_{K}(r))^n$ with $\C^{rn}$ by
  mapping a complex polynomial $\sum_{i=0}^{r-1}a_i t ^i$ in
  $\cO_{K}(r)$ to the tuple $(a_i)_{i=0}^{r-1}$ in $\C^r$.
\end{Prop}
\begin{proof}[Proof of Proposition \ref{prop:constr}]
  By the mentioned quantifier elimination result for our structure
  with three sorts $K$, $\C$ and $\Z$ in the language
  $\cL_{\rm an}^\ac$, the set $X\subset K^n$ is given by a quantifier
  free formula $\varphi(x)$ in the language $\cL_{\rm an}^\ac$, with
  free variables $x$ running over $K^n$.  We proceed by induction on
  the number of occurrences of the function symbol $(\cdot)^{-1}$ in
  the quantifier free formula $\varphi(x)$.  By the form of quantifier
  free formulas (see Theorem 4.2 of \cite{CLR}), $\varphi(x)$ is
  equivalent to a finite Boolean combination of conditions on
  $x\in K^n$ of the form
  \begin{enumerate}[label = \textup{(\alph*)}, ref = (\alph*)]
  \item\label{con1} $\ac(f(x)) = 1$, 
  \item\label{con2} $f(x) = 0$, 
  \item\label{con3} $\ord(f(x))\geq 0$, 
  \item\label{con4} $\ord(f(x)) \equiv 0 \bmod \lambda$.
  \end{enumerate}
  for some $\cL_{\rm
    an}^\ac$-terms $f$ and some integers
  $\lambda>0$.  By the definition of
  $X(r)$, it is enough to prove the proposition when
  $\varphi$ itself has one of the above mentioned forms \ref{con1}, \ref{con2}, \ref{con3}, or \ref{con4}.
  Write $x\in (\cO_{K}(r))^n$ as $(\sum_{\ell=0}^{r-1}
  a_{j\ell}t^\ell)_{j=1}^n$ and write $f(x)$ as $\sum_{s\in\Z}
  f_{s}(a)t^s$, with $a=(a_{j\ell})_{j,\ell}$ and functions
  $f_{s}$ on $\C^{rn}$.

  First suppose that the term
  $f$ does not involve field inversion.  In this case we may suppose that
  $f$ is a finite composition of
  $K$-multiples of restricted analytic mappings, that is, $f= \hat
  f_k\circ \hat f_{k-1}\circ \ldots\circ \hat
  f_1(x)$, where each $\hat
  f_i$ is a map whose component functions are elements of $\cO_K\langle
  x\rangle\otimes_{\cO_K}
  K$, namely, for each $i$ there are restricted analytic functions $
  f_{ij}^|$ for $j=1,\ldots, a_i$ and $\lambda_{ij}\in
  K$ such that $\hat f_i(z)= (\lambda_{ij}
  f^|_{ij}(z))_{j}$, where $j=1,\ldots,a_i$, $z\in
  K^{a_{i-1}}$, and where $a_i$ is the arity of $\hat
  f_{i+1}$ if $i<k$ and with $a_k=1$. (Indeed, since $X(r)\subset
  \cO_K^n$, the global ring operations on
  $K$ are irrelevant in the presence of all restricted analytic functions). Write
  $g_i$ for $\hat f_i\circ \hat f_{i-1}\circ \ldots\circ \hat
  f_1(x)$ for each
  $i=1,\ldots,k$.  Let us finish this case by induction on
  $k$. 
  In fact we will prove at the same time an additional statement by induction on
  $k$: there is an integer $N>0$ such that for $x\in
  \cO_K(r)^n$, one either has $f(x)=0$ or $-N\leq \ord f(x) \leq
  N$, and, for any tuple of integers $\hat \nu = (\hat
  \nu_{ij})_{i,j}$, there are polynomials $p_{\hat \nu,
    s}$ such that $f_s(a)=p_{\hat \nu, s}(a)$ whenever $x\in
  \cO_K(r)^n$ satisfies $\ord g_{ij}(x) = \hat
  \nu_{ij}$ for all $i,j$ and where
  $a$ is still such that $x=(\sum_{\ell=0}^{r-1}
  a_{j\ell}t^\ell)_{j=1}^n$.

  If $k=1$, then clearly each of the $f_{s}$ 
  is a polynomial in the tuple $a$ (no need to specify $\hat \nu$),
  and, there is an integer $M>0$ such that $t^Mf $ lies in
  $\cO_K\langle x\rangle$.  By the Noetherianity of any polynomial
  ring in finitely many variables over $\C$, there exists $N>M$ such
  that $\ord f(x) > N$ implies $f(x)=0$ for $x\in (\cO_{K}(r))^n$.
  It is thus sufficient to show the constructibility in
  $x\in (\cO_{K}(r))^n$ of the following conditions:
  \begin{enumerate}[label = \textup{(\arabic*)}, ref = (\arabic*)]
  \item\label{con1k} $\ac(f(x)) = 1 \wedge
    \ord(f(x))=\mu$, 
  \item\label{con2k} $\ord (f(x))>N$, 
  \item\label{con3k} $\ord(f(x))=\mu$, 
  \end{enumerate}
  for integers $\mu$ with $-N \leq \mu \leq
  N$. Each of these cases is straightforward, for example, condition \ref{con1k} on $x\in
  (\cO_{K}(r))^n$ is equivalent to
$$
f_{\mu}(a) =1 \wedge_{s=-N}^{\mu-1} f_{s}(a)=0,
$$
which is clearly a constructible condition on $a\in\C^{nr}$ (where
$a$ corresponds to $x$ as above). Also
$N$ is as desired for the additional statement. This finishes the case that
$k=1$.

The case of general $k>1$ goes as follows.  By induction on
$k$ we have constructibility and the additional statement for
$g_{k-1}$ for some $N>0$ and all $\hat \nu$. Let us fix a tuple $\hat
\nu$ with $-N\leq \hat \nu_{ij}\leq
N$ for all $i$ and $j$ and a maximal consistent set
$W$ of conditions on $x\in (\cO_{K}(r))^n$ of the form $\ord g_{ij}(x)
= \hat
\nu_{ij}$ or of the form $g_{ij}(x)=0$, where
$i<k$. Now the case that $x\in
(\cO_{K}(r))^n$ satisfies the conjunction of
$\varphi(x)$ with the conditions from
$W$ can be finished as the argument for
$k=1$ (including the additional statement), by using again the Noetherianity of polynomial rings.  This finishes the case that the term
$f$ does not involve the function symbol for field inversion.

Next, suppose that $f$ contains a sub-term $g$ which is of the form
$g_0^{-1}$ for some term $g_0$ which does not involve field inversion.
Choose $N>0$ for $g_0$ as given by the base case of our induction and
its additional statement, as shown above.
For each integer $\nu$ with $-N \leq \nu \leq N$, consider the
condition $B_{\nu} (x,\xi)$ on $(x,\xi)\in K^{n+1}$ being the
conjunction of $\ord(\xi)=0$ with
$$
\ord g_0(x) = \nu\wedge \ac(\xi g_0(x)) = 1
$$
and with $\varphi(x)$, where we recall that $\varphi$ is the condition
\ref{con1}, \ref{con2}, \ref{con3}, or
\ref{con4}. 

Note that there is an $\cL_{\rm an}^\ac$-term $h(x,\xi)$, with
$(x,\xi)$ running over $K^{n+1}$, which does not involve field
inversion and such that for $x\in \cO_K^{n}$ and
$\xi\in\C^\times \subset \cO_K^\times $ with
$\ord g_0(x) = \nu \wedge \ac(\xi g_0(x)) = 1$, we can write
$$
g(x) =\frac{1}{g_0(x)} = \xi t^{-\nu}\frac{1}{1-(1-\xi t^{-\nu} g_0(x))} =
h(x,\xi).  
$$
Indeed, $h$ comes from developing $1/(1-m)$ as a power series in $m$
running over the maximal ideal $t\cO_K$ of $\cO_K$, and evaluating at
$m=1-\xi t^{-\nu} g_0(x)$.
Let us now, inside $f$, replace the subterm $g(x)$ by
$h(x,\xi)$. Then, we see by our ongoing induction on the number of
occurrences of the function symbol $(\cdot)^{-1}$, that condition
$B_{\nu} (x,\xi)$ on $(x,\xi)\in\cO_K(r)^{n+1}$ yields a constructible
subset $A_\nu$ of $\C^{r(n+1)}$. Hence, also the set $A'_\nu$
consisting of $(x,\xi)$ in $A_\nu$ with moreover
$\xi\in\C^\times \times \{0\}^{r-1} \subset \C^r$ is constructible.
By the Chevalley-Tarski Theorem, the image $A''_\nu$ of $A'_\nu$ under
the coordinate projection to $\C^{nr}$ is also constructible, and,
$A''_\nu$ equals the set of $x\in \cO_K(r)^{n}$ satisfying
$\varphi(x)\wedge \ord g_0(x) = \nu$, where $\varphi$ is still
\ref{con1}, \ref{con2}, \ref{con3}, or \ref{con4}.

Finally, consider the condition $B^0(x)$ which is the conjunction of
$g_0(x)=0$ and the condition $\varphi'(x)$ obtained from $\varphi(x)$
by replacing the sub-term $g(x)$ by zero in the term $f(x)$.  By
induction on the number of occurrences of the function symbol
$(\cdot)^{-1}$, 
the condition $B^0(x)$ on $x\in\cO_K^{n}(r)$ is Zariski constructible
and the extra statement also holds.

By construction, for $x\in \cO_K^{n}(r)$, $x$ lies in $X(r)$ if and
only if $B^0(x)$ holds or $x\in A''_\nu$ for some $\nu$ with
$-N \leq \nu \leq N$. This finishes the proof of the proposition.
\end{proof}


\begin{Rem}
  In fact, we will prove Theorem \ref{thm:finite} more generally for
  any $\cL_{\rm an}^\ac$-definable set $X\subset
  K^n$. 
\end{Rem}

\begin{proof}[Proof of Theorem \ref{thm:finite}]
  We will prove the theorem for a subset $X\subset K^n$ which is
  $\cL_{\rm an}^\ac$-definable. Let $r>0$ be an integer.  By
  Proposition \ref{prop:constr}, $X(r)$ is a constructible subset of
  $\C^{rn}$.  We prove the finiteness of $(X^{\rm trans})(r)$ by
  induction on the dimension $\ell$ of $X(r)$. If $\ell=0$, there is
  nothing left to prove since $X(r)$ is finite in this case and since
  $(X^{\rm trans})(r)\subset X(r)$.
  If $\ell>0$, write $A$ for $X(r)$ and choose an algebraic family of
  algebraic (locally closed) curves $C_v\subset \AA_\C^{rn}$ for
  $v\in V\subset \AA_\C^s$ for some $s\geq 0$ such that the union of
  the sets $C_v(\C)$ over $v\in V(\C)$ equals $A(\C)\setminus F$,
  where $F$ is a finite set. Such a family clearly exists. For each
  $v\in V(\cO_K)$, let $S_v$ be the image of $C_v(\cO_K)$ under the
  map $p:\cO_K^{rn}\to \cO_K^n$ sending $(x_{jk })$ to
  $(\sum_{k =0}^{r-1} x_{jk }t^k)_{j=1}^n$. Clearly $S_v$ is
  semi-algebraic and of dimension $1$ for each $v\in V(\cO_K)$, since
  $S_v$ is infinite and equal to the image under a semi-algebraic
  function of an algebraic curve.  (See e.g.~Section 3 of
  \cite{Dries-dim}, or, the dimension theory of \cite{CLb} and Theorem
  6.3.7 from
  \cite{CLip}.) 
  For each $v\in V(\cO_K)$, let $S_v'$ be the subset of $S_v\cap X$
  consisting of $x$ such that $S_v\cap X$ is locally of dimension $1$
  at $x$. Let $X'$ be $X\setminus (\cup_{v\in V(\cO_K)} S_v')$. Then
  clearly $X'$ is $\cL_{\rm an}^\ac$-definable.  Moreover, by
  construction we have that $X'(r)$ is of dimension less than $\ell$
  (as constructible subset of $\C^{rn}$). Indeed, $C_v(\C)\cap X'(r)$
  is finite for any $v\in V(\C)$, since $(S_v\cap X)\setminus S'_v$ is
  finite for each $v\in V(\cO_K)$. (Here we have used that the subset
  of points $x$ in a semi-algebraic set $S$ of local dimension $0$ at
  $x$ is finite, see again \cite{Dries-dim} or \cite{CLb}.)
  Since clearly $X^{\rm trans}(r)$ is contained in
  $X'{}^{\rm trans}(r)$, we can replace $X$ by $X'$ and thus we are
  done by induction on $\ell$.
\end{proof}

\begin{proof}[Proof of Proposition \ref{thm:fast-in-r}]
  Let $N_k$ be a strictly increasing sequence and consider the
analytic function $\cO_K\to \cO_K$ given by the following
  converging power series $f$ in $K[[x]]$
  \begin{equation}
    f(x) = \sum_{i> 0} t^{i} P_{N_i}, \qquad \mbox{where } P_k = x^k \prod_{j=1}^k(x-j).
  \end{equation}
  Let $X\subset \cO_K^2$ denote the analytic set $y=f(x)$ for $x\in\cO_K$. For each $i>0$ and  $j=1,\ldots,N_i$, note that $f(j)$ is a polynomial in $t$ of degree less than $i$
  and therefore
  \begin{equation}\label{eq:Xr-growth}
    \#X(i)\ge N_i.
  \end{equation}
  On the other hand,
  \begin{equation}\label{eq:f-supp}
    \supp_x f(x) \subset \cup_i [N_i,2N_i] \quad\text{where}\quad \supp_x(\sum c_j x^j)=\{j:c_j\neq0\}.
  \end{equation}
  If we choose $N_i$ to grow sufficiently quickly (e.g.~if for every
  $d\in\N$ eventually $N_i>2dN_{i-1}$) then~\eqref{eq:f-supp} implies
  that $f(x)$ is transcendental. Indeed, if $f$ is algebraic of degree
  $d$ then by the B\'ezout theorem its order of contact with its
  $2N_{i-1}$-Taylor approximation should not exceed $2dN_{i-1}$.
  Choosing $N_i$ in this manner we obtain from~\eqref{eq:Xr-growth} the
  lower-bound  $\# X^\trans(i) = \#X(i) \ge N_i$, which finishes the proof.
%
%
%
\end{proof}


\subsection*{A mixed characteristic variant}\label{sec:Witt}

In this subsection, let $L$ be $\Q_p^{\rm unram}$, the maximal unramified
field extension of $\Q_p$ for some prime $p$.  Let
$\tau: \FF_p^{\rm alg } \to \cO_L$ be the Teichm\"uller lifting,
namely the unique multiplicative section of the natural projection map
$\cO_L\to \FF_p^{\rm alg }$ with $\cO_L$ the valuation ring of $L$,
and where $\FF_p^{\rm alg }$ is an algebraic closure of $\FF_p$.

We use analogous notation and definitions for $L$ as above for $K$,
with the following natural adaption to define $\cO_{L}(r)$.  For any
integer $r\geq 1$, denote by $\cO_{L}(r)$ the subset of $\cO_L$
consisting of elements of the form
$\sum_{i=0}^{r-1} \tau(a_i) p^i$. These elements can informally be
seen as polynomials in $p$ with coefficients in
$\tau(\FF_p^{\rm alg })$ and degree less than $r$. We identify
$\cO_{L}(r)$ with $(\FF_p^{\rm alg })^r$ by sending
$\sum_{i=0}^{r-1} \tau(a_i) p^i$ to the tuple $(a_i)_i$.

\begin{Thm}[Finiteness in mixed characteristic]\label{thm:finite:L}
  Let $X\subset L^n$ be $\cL_{\rm an}^L$-definable. Then, for each
  $r>0$, $X(r)$ is a constructible subset of $(\FF_p^{\rm alg })^{rn}$
  and
$$
(X^{\rm trans})(r)
$$
is a finite set.
\end{Thm}
\begin{proof}
  The proof is as for Proposition \ref{prop:constr} and Theorem
  \ref{thm:finite}, using the usual operations on Witt-vectors, with
  natural adaptations, in particular putting $p$ in the role of
  $t$. For the first part, one adapts the proof of Proposition
  \ref{prop:constr} by using also higher order angular component maps
  and identifying $\cO_L/p^k\cO_L$ with $(\FF_p^{\rm alg })^{k}$ in
  order to have quantifier elimination. A higher order angular
  component map is a map $L\to \cO_L/p^k\cO_L$ for some integer $k>0$
  sending nonzero $x$ to $p^{-\ord x}x\bmod p^k\cO_L$ and zero to
  zero.  
  For the second part, one repeats the proof of Theorem
  \ref{thm:finite} where one takes a lift of the family of curves
  $C_v\subset \AA_{\FF_p^{\rm alg } }^{rn}$ to a family of curves
  $C_v$ in $\AA_L^{rn}$ by applying $\tau$ to the coefficients of the
  defining polynomials.
\end{proof}

\section{Non-archimedean Wilkie type Conjecture for Pfaffian and Noetherian varieties}
\label{sec:wilkie-results}

Our goal is to prove an analog of the Wilkie
conjecture over the field of convergent Laurent series $\C(\!\{t\}\!)$ for
varieties defined using Pfaffian and Noetherian equations.  We begin
by recalling these notions. Denote by $\vx=(\vx_1,\ldots,\vx_n)$ a
system of coordinates on $\C^n$.

\begin{Def}[Pfaffian and Noetherian functions]\label{def:pfaff-noether}
  A \emph{Pfaffian chain} of order $\ell$ and degree $\alpha$ at
  $(\C^n,0)$ is a sequence of holomorphic functions
  $\phi_1,\ldots,\phi_\ell:(\C^n,0)\to\C$ satisfying a triangular
  system of differential equations
  \begin{equation}\label{eq:pfaff}
    \d \phi_j = \sum_{i=1}^n P_{i,j}(\vx,\phi_1(\vx),\ldots,\phi_j(\vx))\d \vx_i, \qquad j=1,\ldots,\ell
  \end{equation}
  where $P_{i,j}$ are polynomials of degrees not exceeding $\alpha$.

  A Noetherian chain is defined similarly by dropping the
  triangularity condition, i.e. replacing~\eqref{eq:pfaff} by
  \begin{equation}\label{eq:noetherian}
    \d \phi_j = \sum_{i=1}^n P_{i,j}(\vx,\phi_1(\vx),\ldots,\phi_\ell(\vx))\d \vx_i, \qquad j=1,\ldots,\ell
  \end{equation}
\end{Def}

Given such a Pfaffian (resp. Noetherian) chain, a germ
$f:(\C^n,0)\to\C$ of the form
$f(\vx)=P(\vx,\phi_1(\vx),\ldots,\phi_\ell(\vx))$ where $P$ is a
polynomial of degree not exceeding $\beta$ with coefficients in $\C$
is called a \emph{Pfaffian function} (resp. Noetherian) of order $\ell$
and degree $(\alpha,\beta)$.

Let $K=\C(\!\{t\}\!)$ or $K=\C(t)$.  Let
$\Omega=(\C^n_{\vx}\times\C_t,0)$ denote the germ of
$\C^n_{\vx}\times\C_t$ at the origin.  If $P$ is a polynomial over the
field $K$ then we will say that $f:\Omega\to\C$ of the form
$f(\vx,t)=P(\vx,\phi_1(\vx),\ldots,\phi_\ell(\vx))$ as above is
Pfaffian (resp. Noetherian) over $K$.  If $K=\C(t)$, we also define
the $t$-degree of $f$ by considering the total degree of $P$ as a
polynomial over $\C$ in $n+1$ variables, where we treat $t$ as an
additional variable and clearing the denominators.

\begin{Rem}
  In contrast to Section~\ref{sec:finiteness}, note that we now
  restrict to $K$ given by \emph{convergent} rings. We also restrict,
  in Definition~\ref{def:pfaff-noether}, to systems of differential
  equations with coefficients independent of $t$.

  The former restriction is a technical convenience meant to give a
  situation more closely analogous to the analytic deformations
  studied in \cite{Gab:MultPfaffian} and \cite{me:deflicity}. We
  believe it is likely one can prove analogous results in the Pfaffian
  case also for formal deformations. The latter restriction is more
  serious: in both the Pfaffian and the Noetherian contexts, the
  existing results produce bounds when one deforms a system of
  Pfaffian or Noetherian equations with a \emph{fixed} chain, but not
  when one deforms the chain itself.
\end{Rem}

Denote by $\pi_x$ (resp $\pi_t$) the projection from $\Omega$ to the
$\C^n$ (resp. $\C$) factor. We will say that an analytic germ
$X\subset\Omega$ is \emph{flat} if it is flat with respect to the
projection $\pi_t$, i.e. if $X$ has no components contained in the
fiber $\pi_t^{-1}(0)$. We will identify a $K$-variety in $(\A^n_K,0)$
with the flat holomorphic variety in $\Omega$ defined by identifying
$t$ with the second factor in $\C^n\times\C$ and removing any non-flat
components. This will allow us to apply results from the usual theory
of Pfaffian varieties over $\C$ to the study of Pfaffian varieties
over $K$.

\begin{Def}[Pfaffian and Noetherian varieties]\label{def:pffaf-variety}
  If $f_1,\ldots,f_k:\Omega\to\C$ are Pfaffian (resp. Noetherian) over
  $K$ with degree $(\alpha,\beta)$ and a common Pfaffian chain of
  order $\ell$, we define $V(f_1,\ldots,f_k)\subset\Omega$ to be
  the analytic germ obtained from $\{f_1=\cdots=f_k=0\}$ by removing any
  components contained in the fiber $t=0$. We call a germ obtained in
  this manner a Pfaffian (resp. Noetherian) variety over $K$. Note
  that since we remove components over $t=0$ there will be no harm in
  assuming from the start that the coefficients of the polynomials
  defining $f_1,\ldots,f_k$ are in fact holomorphic at $t=0$, and the
  corresponding functions are in fact holomorphic germs.
\end{Def}

\begin{Ex}
  Let $n=2$ and consider the Noetherian chain whose elements are given
  by $x,y$. The two equations $f_1=x$ and $f_2=x-ty$ define the
  isolated point $x=y=0$ over generic $t$, but over $t=0$ the entire
  $y$-axis appears as an extra component. In our notation $V(f_1,f_2)$
  will be the analytic germ defined by $x=y=0$.
\end{Ex}

Below we will suppose that a Pfaffian chain has been fixed, and in our
asymptotic notation we will allow the constants to depend on
$\alpha,n,\ell$ (we note that all constants can be explicitly
computed). We will refer to $\beta$ in
Definition~\ref{def:pffaf-variety} as the complexity of the Pfaffian
(resp. Noetherian) variety. If $K=\C(t)$ we   define the
notion of $t$-complexity by replacing degrees with $t$-degrees.

Let $X\subset\Omega$ be a flat analytic germ. If $p:(\C,0)\to(\C^n,0)$
is an analytic germ, we denote by $\tilde p:(\C,0)\to\Omega$ the map
$t\to(p(t),t)$. If every coordinate of $p$ is a polynomial, of degree
less than $r$ for some $r>0$, then we write $\deg p<  r$ (otherwise we set $\deg p=\infty$). 

\begin{Def}
  We denote
  \begin{equation}
    X(r):=\{ p:(\C,0)\to(\C^n,0) : \Im \tilde p\subset X,\  \deg p< r\}.
  \end{equation}
  More generally, if $\vg:(\C^n,0)\to(\C^k,0)$ is holomorphic we
  denote
  \begin{equation}
    X(\vg,r):=\{ \Im \tilde p : \Im \tilde p\subset X,\ \deg \vg(p)< r \}.
  \end{equation}
\end{Def}

Our first main result is the following analog of the Wilkie conjecture
for Pfaffian varieties over $\C(\!\{t\}\!)$.

\begin{Thm}\label{thm:pfaffian-wilkie}
  Let $X\subset\Omega$ be a Pfaffian variety over $\C(\!\{t\}\!)$ of
  complexity $\beta$ and $r\in\N$. Then there exists a collection
  $\{W_\eta\}_\eta$ of irreducible algebraic varieties over
  $\C(t)$ such that $W_\eta\subset X$ as germ at the origin and
  $X(r)\subset\cup_\eta W_\eta(r)$. Moreover,
  \begin{equation}
    \#\{W_\eta\} = \poly(\beta,r), \qquad \deg W_\eta = \poly(\beta,r).
  \end{equation}
\end{Thm}

Theorem~\ref{thm:pfaffian-wilkie} is analogous to the Wilkie
conjecture combined with Pila's ``blocks'' formalism, where each
algebraic $W_\eta$ can be thought of as a block. In particular any
positive dimensional $W_\eta$ lies by definition in $X^\alg$, and
\begin{equation}
  X^\trans(r) = \bigcup_{\eta:\dim W_\eta=0} W_\eta(r).
\end{equation}
By irreducibility $\#W_\eta(r)\le1$ for the zero-dimensional
$W_\eta$, and one concludes that $\#X^\trans(r)=\poly(r,\beta)$.

Our second main result is the following analog of the Wilkie
conjecture for Noetherian varieties over $\C(t)$. We are unable to
prove the same result over $\C(\!\{t\}\!)$ in the Noetherian category due
to certain limitations in the available complexity estimates for the
Noetherian category (see Fact~\ref{fact:noetherian-bound} and the
following discussion). On the other hand, we remark that many of the
classical transcendental functions involved in applications of the
Pila-Wilkie theorem do lie in the Noetherian (and not in the Pfaffian)
category, and the limitation of algebraic dependence on the variable
$t$ does not seem to be overly restrictive.

\begin{Thm}\label{thm:noetherian-wilkie}
  Let $X\subset\Omega$ be a Noetherian variety over $\C(t)$ of
  $t$-complexity $\beta$ and $r\in\N$. Then there exists a collection
  $\{W_\eta\}_\eta$ of irreducible algebraic varieties over
  $\C(t)$ such that $W_\eta\subset X$ as germ at the origin and
  $X(r)\subset\cup_\eta W_\eta(r)$. Moreover,
  \begin{equation}
    \#\{W_\eta\} = \poly(\beta,r), \qquad \deg W_\eta = \poly(\beta,r).
  \end{equation}
\end{Thm}

The proofs of
Theorems~\ref{thm:pfaffian-wilkie}~and~\ref{thm:noetherian-wilkie} are
given in Section~\ref{sec:wilkie-proofs}, after developing some
preliminary material in Section~\ref{sec:weierstrass-polydiscs}.
\section{Weierstrass polydiscs and interpolation}
\label{sec:weierstrass-polydiscs}

\subsection{Analytic germs and Weierstrass coordinates}

\begin{Def}
  Let $X\subset\Omega$ be a flat analytic germ of pure dimension
  $m+1$. Let $\vx:=\vz\times\vw:\C^n\to\C^p_{\vz}\times\C^{n-p}_{\vw}$ be a unitary linear  map. We
  will say that $\vz\times\vw$ are \emph{Weierstrass coordinates} for
  $X$ if $p=m$ and the projection
  $\pi_z\times\pi_t:X\to(\C^m\times\C,0)$ is finite. We denote by
  $e(X,\vx)$ the degree of this projection, i.e. the number of points
  (counted with multiplicities) in the fiber of any point
  $p\in(\C^m\times\C,0)$.
\end{Def}

\begin{Lem}\label{lem:weierstrass-coordinates}
  Weierstrass coordinates exist for every flat analytic germ
  $X\subset\Omega$ of pure dimension $m+1$.
\end{Lem}
\begin{proof}
  Since $X$ is flat the fiber $X_0$ over $t=0$ has pure dimension $m$.
  Let $\vx=\vz\times\vw$ be unitary coordinates and
  $\Delta_z\times\Delta_w$ a Weierstrass polydisc (in the sense of
  \cite{me:rest-wilkie}) for $X_0$, i.e.
  $X_0\cap(\bar\Delta_z\times\partial\Delta_w)=\emptyset$. Since $X$
  is closed, for a sufficiently small disc $D_t\subset(\C,0)$ we also
  have
  $X\cap(\bar\Delta_z\times\partial\Delta_w\times\bar D_t)=\emptyset$.
  Then $(\Delta_z\times D_t)\times(\Delta_w)$ is a Weierstrass
  polydisc for $X$, and in particular the projection
  $\pi_z\times\pi_t:X\to(\C^m\times\C,0)$ is finite.
\end{proof}

\begin{Prop}\label{prop:analytic-decomp}
  Let $X\subset\Omega$ be a flat analytic germ of pure dimension $m+1$
  and $\vx$ Weierstrass coordinates for $X$, and set $\nu:=e(X,\vx)$.
  Then for any $f\in\cO(\Omega)$ there exists a function
  \begin{equation}
    P\in\cO_0(\C^m\times\C)[w], \qquad \deg_{\vw_i} P\le \nu-1,\quad i=1,\ldots,n-m
  \end{equation}
  such that $f\rest X\equiv P\rest X$, and where $\cO_0$ stands for the holomorphic germs at zero.
\end{Prop}
\begin{proof}
  The claim follows from \cite[Proposition~7]{me:rest-wilkie} applied
  to a Weierstrass polydisc in the $\vx,t$ coordinates as constructed
  in the proof of Lemma~\ref{lem:weierstrass-coordinates}.
\end{proof}

\subsection{Interpolation determinants}
In this section we give a non-archimedean analogue of the interpolation determinant method of Bombieri-Pila \cite{BP}.

Let $X\subset\Omega$ be a flat analytic germ of pure dimension $m+1$
and $\vx$ Weierstrass coordinates for $X$, and set $\nu:=e(X,\vx)$.
Let $\vp=(p_1,\ldots,p_\mu)$ with $p_j:(\C,0)\to\C^n$  and
$\vf=(f_1,\ldots,f_\mu)$ with $f_j\in\cO(\Omega)$. We define the
\emph{interpolation determinant} of $\vf$ and $\vp$ to be
$\Delta(\vf,\vp):=\det (f_i(\tilde p_j))$.

\begin{Prop}\label{prop:order-lower-bd}
  Suppose that $\Im \tilde p_j\subset X$ for all $j$. Then
  \begin{equation}\label{eq:order-lb}
    \ord_t \Delta(\vf,\vp) \ge C_n \nu^{-(n-m)/m} \mu^{1+1/m}
  \end{equation}
  where the constant $C_n$ depends only on $n$.
\end{Prop}
\begin{proof}
  Expand each $f_j$ using Proposition~\ref{prop:analytic-decomp}, and
  then expand the determinant $\Delta(\vf,\vp)$ by linearity in
  each column over $\C(\!\{t\}\!)$. Each term of order $k$ in the
  $\vx$ variables, when evaluated at $\tilde p_j$, has order at least
  $k$ in $t$. Moreover, in the decomposition of each $f_i$ there are
  fewer than $\nu^{n-m} k^{m-1}$ linearly independent terms of each
  order $k$. It follows that
  \begin{equation}\label{eq:order-lb-1}
    \ord_t\Delta(\vf,\vp) \ge \sum_{k=0}^b (\nu^{n-m} k^{m-1})\cdot k
  \end{equation}
  where
  \begin{equation}
    \mu \ge \sum_{k=0}^b \nu^{n-m} k^{m-1}.
  \end{equation}
  Then we have $b\sim (\mu/\nu^{n-m})^{1/m}$, and plugging in
  to~\eqref{eq:order-lb-1} gives \eqref{eq:order-lb}.
\end{proof}

\subsection{Polynomial interpolation determinants}

Fix $d\in N$, and let $\mu$ denote the dimension of the space of
polynomials of degree at most $d$ in $m+1$ variables. Note that
$\mu\sim d^{m+1}$.

For $\vg$ an $m+1$-tuple of functions and $\vp$ a $\mu$-tuple of
points, we define the \emph{polynomial interpolation determinant}
$\Delta^d(\vg,\vp)):=\Delta(\vf,\vp)$ where $\vf$ is the tuple of all
monomials of degree at most $d$ in the coordinates of $\vg$.

\begin{Prop}\label{prop:deg-upper-bd}
  Let $\vg:(\C^n,0)\to(\C^{m+1},0)$ and suppose
  \begin{equation}
    p_i\in X(\vg,r+1), \qquad j=1,\ldots,\mu.
  \end{equation}
  Then
  \begin{equation}
    \deg\Delta^d(\vg,\vp) \le E_n d^{m+2} r.
  \end{equation}
  where $E_n$ is some constant depending only on $n$.
\end{Prop}
\begin{proof}
  There are $\mu\sim d^{m+1}$ columns in the matrix defining
  $\deg\Delta^d(\vg, \vp)$ and each of them has degree at most
  $dr$ in $t$.
\end{proof}

\begin{Cor}\label{cor:hypersurface-select}
  Let $\vg:(\C^n,0)\to(\C^{m+1},0)$. Then there exists a constant
  $A_n$ depending only on $n$ such that if
  \begin{equation}\label{eq:hypersurface-select}
    d > A_n \nu^{n-m} r^m
  \end{equation}
  then $X(\vg,r+1)$ is contained in the zero locus of a polynomial
  $P\in\C(\!\{t\}\!)[\vg]$ of degree at most $d$.
\end{Cor}
\begin{proof}
  By elementary linear algebra, if $\Delta^d(\vg,\vp)$ vanishes
  for each $\vp_i\in X(\vg,r+1)$ then the conclusion of the corollary
  follows. To see that this indeed happens, we use
  Proposition~\ref{prop:order-lower-bd} and
  Proposition~\ref{prop:deg-upper-bd}, and note that the order at zero
  of a polynomial cannot exceed its degree. Therefore, unless
  $\Delta^d(\vg, \vp)$ vanishes we have
  \begin{equation}
     C_n \nu^{-(n-m)/m} \mu^{1+1/m} \le \ord_t\Delta(\vg, \vp) \le
    \deg\Delta^d(\vg, \vp) \le E_n d^{m+2} r.
  \end{equation}
  Since $\mu\sim d^{m+1}$ this gives
  \begin{equation}
    C_n d^{1/m} \le E_n \nu^{(n-m)/m} r,
  \end{equation}
  and for $d$ as in~\eqref{eq:hypersurface-select} we indeed obtain a
  contradiction.
\end{proof}

\section{Proofs of the main Theorems}
\label{sec:wilkie-proofs}

\subsection{The Pfaffian case}

We will use the following result of Gabrielov.

\begin{Fact}[\protect{\cite[Theorem~2.1]{Gab:MultPfaffian}}]\label{fact:pfaffian-bound}
  Let $f_1,\ldots,f_n:\Omega\to\C$ be Pfaffian over $\C(\!\{t\}\!)$ of
  degree $(\alpha,\beta)$ over a common Pfaffian chain of order $\ell$
  and let $X=V(f_1,\ldots,f_n)$. Then the number of isolated points
  (counted with multiplicities) in the fiber $X_t$ converging to the
  origin as $t\to0$ is bounded by $O(\beta^{n+\ell})$ where the
  constants depend only on $\alpha,n,\ell$.
\end{Fact}

We now deduce a general result on Weierstrass coordinates for Pfaffian
varieties over $\C(\!\{t\}\!)$. For an analytic germ $X\subset\Omega$
we denote by $X^{\le m}$ the union of the irreducible components of
$X$ having dimension $m$ or less (and similarly for $X^{m}$).

\begin{Thm}\label{thm:pfaff-weierstrass}
  Let $X\subset\Omega$ be a Pfaffian variety over $\C(\!\{t\}\!)$ of
  complexity $\beta$ and $1\le m\le n$. Then there exists an analytic
  germ $Z\subset\Omega$ of pure dimension $m$ satisfying
  $X^{\le m}\subset Z$ and Weierstrass coordinates $\vx$ for $Z$ such
  that $e(Z,\vx)=O(\beta^{n+\ell})$.
\end{Thm}
\begin{proof}
  Let $\tilde f_1,\ldots,\tilde f_{n+1-m}$ be $n+1-m$ generic linear
  combinations of the Pfaffian functions defining $X$, and set
  $\tilde X:=V(\tilde f_1,\ldots,\tilde f_{n+1-m})$ and
  $Z=\tilde X^m$. It is easy to see that for a sufficiently generic
  choice, every component of $X^{\le m}$ is contained in a component
  of $\tilde X$ of dimension $m$ and hence $X^{\le m}\subset Z$.

  Let $\vx$ be a set of Weierstrass coordinates for $Z$ and set
  $\pi^Z:=\pi_z\times\pi_t \rest Z$. Denote by $Z_b$ the set of
  components of $\tilde X$ of dimension greater than $m$. Then
  $Z\cap Z_b$ has dimension strictly smaller than $m$, and
  $Y:=\pi^Z(Z\cap Z_b)$ is a (strict) analytic germ in
  $(\C^m\times\C,0)$. It follows that for a generic choice of the
  vector $v\in\C^m$, the line $\C\cdot (v,1)$ meets $Y$ only at the
  origin. Thus, a fiber of the map $\pi^Z$ over any point
  $(v\cdot t,t)$ with $t\in(\C,0)$ and $t\neq0$ consists of
  $\nu=e(Z,\vx)$ isolated points in $Z\setminus Z_b$ which converge to
  the origin as $t\to0$. Each such isolated point is an isolated
  solution of the system
  \begin{equation}
    \{\tilde f_1=\cdots=\tilde f_{n+1-m}=0, \quad \vz_1=v_1\cdot t,\cdots,\vz_m=v_m\cdot t\}
  \end{equation}
  and the bound on $\nu$ thus follows from
  Fact~\ref{fact:pfaffian-bound}.
\end{proof}

The following proposition gives the basic induction step for the proof
of the Wilkie conjecture over $\C(\!\{t\}\!)$.

\begin{Prop}\label{prop:wilkie-indcution}
  Let $X\subset\Omega$ be a Pfaffian variety over $\C(\!\{t\}\!)$. Let
  $W\subset\Omega$ be an irreducible algebraic variety over
  $\C(\!\{t\}\!)$ of dimension $k$ over $\C(\!\{t\}\!)$, and suppose
  that $X\subset W$ and $X\neq W$.

  Then for any $r>0$ there exists an algebraic hypersurface
  $H\not\supset W$ over $\C(\!\{t\}\!)$ of degree
  $O(\beta^{(n+\ell)(n-k+1)}r^{k-1})$ such that
  $X(r)\subset(X\cap H)(r)$.
\end{Prop}

\begin{proof}
  Let $\vg:(\C^n,0)\to(\C^k,0)$ be a linear projection which is
  dominant on $W$. Note that necessarily $\dim_\C X<\dim_\C W=k+1$. By
  Theorem~\ref{thm:pfaff-weierstrass} we choose an analytic germ
  $Z\supset X$ of pure dimension $k$ (over $\C$) and Weierstrass
  coordinates $\vx$ for $Z$ such that
  $\nu:=e(Z,\vx)=O(\beta^{n+\ell})$. Then by
  Corollary~\ref{cor:hypersurface-select} we may choose a polynomial
  $P\in\C(\!\{t\}\!)[\vg]$ of degree at most $O(\nu^{n+1-k}r^{k-1})$
  such that
  \begin{equation}
    X(r)\subset Z(\vg,r)\subset H, \qquad H=\{P=0\}
  \end{equation}
  and $W\not\subset H$ since $\vg$ was assumed to be dominant on $W$.
\end{proof}

\begin{proof}[Proof of Theorem~\ref{thm:pfaffian-wilkie}]
  First set $X'=X,W'=\A^n_{\C(\!\{t\}\!)}$. If $X'=W'$ then we can
  finish with $\{W_\eta\}=\{W'\}$. Otherwise we may apply
  Proposition~\ref{prop:wilkie-indcution} with $W=W'$ to find a
  collection of $\poly(\beta,r)$ hypersurfaces
  $H_j\subset\A^n_{\C(\!\{t\}\!)}$ such that
  \begin{equation}
    X'(r) \subset \cup_j (X'\cap H_j)(r).
  \end{equation}
  Denote by $\{H'_k\}_k$ the collection of irreducible components of
  the hypersurfaces $H_j$. It will be enough to prove the claim for
  each pair $(X'\cap H'_k,W'\cap H'_k)$ and take the union of all the
  resulting collections $\{W_\eta\}$. For this we repeat the same
  argument as above with $(X',W')$ replaced by this pair.

  Proceeding in this manner for $n$ steps we end up with $W'$ of
  dimension zero over $\C(\!\{t\}\!)$ and $X'\subset W'$. If $X'=W'$
  we can take $\{W_\eta\}=\{W'\}$, and otherwise the intersection
  $X'\cap W'$ is empty and we take $\{W_\eta\}=\emptyset$.
\end{proof}

\subsection{The Noetherian case}

The proof in the Noetherian case is entirely analogous to the proof in
the Pfaffian case, and we leave the detailed derivation for the
reader. The only significant difference is that in this case
Fact~\ref{fact:pfaffian-bound} should be replaced by the following.

\begin{Fact}[\protect{\cite[Theorem~1]{me:deflicity}}]\label{fact:noetherian-bound}
  Let $f_1,\ldots,f_n:\Omega\to\C$ be Noetherian over $\C(t)$ of
  $t$-degree $(\alpha,\beta)$ over a common Pfaffian chain of order
  $\ell$ and let $X=V(f_1,\ldots,f_n)$. Then the number of isolated
  points (counted with multiplicities) in the fiber $X_t$ converging
  to the origin as $t\to0$ is bounded by $\poly(\beta)$ where the
  constants depend only on $\alpha,n,\ell$.
\end{Fact}

Note that, comparing with Fact~\ref{fact:pfaffian-bound}, in
Fact~\ref{fact:noetherian-bound} the bound depends on the $t$-degree
of the equations with respect to $t$. We must therefore restrict to
varieties over $\C(t)$. It has been conjectured by
Gabrielov-Khovanskii \cite{GK:MultNoetherian} that a similar result
should hold without dependence on the $t$-degree, and this conjecture
would indeed imply an analog of Theorem~\ref{thm:pfaffian-wilkie}
without the restriction to $\C(t)$.

\subsection{A final remark}

In new work in \cite{CNV} it is shown that the finiteness result from Theorem \ref{thm:finite}  no longer holds if one weakens the condition of $\cL_{\rm an}^K$-definability to just definability in a hensel-minimal structure. Hensel minimality is a non-archimedean analogue of o-minimality, introduced in \cite{CHR} \cite{CHRV}. Other bounds based on dimension (instead of on counting) are put forward and are shown in the hensel minimal curve case in \cite{CNV}. The higher dimensional case remains open and would represent the non-archimedean analogue of the bounds in o-minimal structures from \cite{PW}.

%
%
%

\bibliographystyle{plain} \bibliography{nrefs}

\end{document}